\documentclass{amsart}
\usepackage{amsthm}
\usepackage{amsmath}
\usepackage{amssymb}
\usepackage{hyperref}
\newtheorem{theorem}{\sc{Theorem:}}[section]

\newtheorem{definition}[theorem]{\sc{Definition:}}

\newtheorem{lemma}[theorem]{\sc{Lemma:}}
\newtheorem{proposition}[theorem]{\sc{Proposition:}}
\newtheorem{corollary}[theorem]{\sc{Corollary:}}

\theoremstyle{remark}
\newtheorem{remark}[theorem]{\sc{Remark:}}

\begin{document}

\title{THE MAGNETIC RAY TRANSFORM ON ANOSOV SURFACES}
\date{}
\author[G. Ainsworth]{GARETH AINSWORTH}

\begin{abstract} Assume \((M,g,\Omega)\) is a closed, oriented Riemannian surface equipped with an Anosov magnetic flow. We establish certain results on the surjectivity of the adjoint of the magnetic ray transform, and use these to prove the injectivity of the magnetic ray transform on sums of tensors of degree at most two. In the final section of the paper we give an application to the entropy production of magnetic flows perturbed by symmetric \(2\)-tensors.
\end{abstract}
\maketitle

\section{Introduction}
\indent \indent Let \((M,g)\) be a closed, oriented Riemannian manifold. Consider the function \(H:TM\rightarrow\mathbb{R}\) given by 
\[H(x,v):=\frac{1}{2}g(v,v),\ \ \ (x,v)\in TM.\]
The geodesic flow on \(TM\) is given by the Hamiltonian flow of the above function with reference to the symplectic structure \(\omega_{0}\) on \(TM\) provided by the pullback, via the metric, of the canonical symplectic form on \(T^{*}M.\) The abstract formulation of a magnetic field imposed on \(M\) is specified by a closed 2-form \(\Omega\). The \textit{magnetic flow}, or twisted geodesic flow, is defined as the Hamiltonian flow of \(H\) under the symplectic form \(\omega\), where \[\omega:=\omega_{0}+\pi^{*}\Omega,\] and \(\pi:TM\rightarrow M\) is the usual projection. Magnetic flows were first studied in \cite{anosov67, arnold61}; for more recent references in relation to inverse problems, see below.\\
\indent We may alternatively think of the magnetic field as being determined by the unique bundle map \(Y:TM\rightarrow TM,\) defined via, \[\Omega_{x}(\xi,\eta)=g(Y_{x}(\xi),\eta),\ \ \ \forall x\in M,\ \forall \xi,\eta\in T_{x}M.\] Note that this implies that \(Y\) is skew-symmetric.
The advantage of this point of view is that it provides a nice description of the generator of the magnetic flow, indeed, one can show that this vector field at \((x,v)\in TM\) is given by \[X(x,v)+Y^{i}_{k}(x)v^{i}\frac{\partial}{\partial v^{k}}.\] Here note that the coefficients of \(Y\) are given by \(Y(\frac{\partial}{\partial x^{j}})=Y^{j}_{i}\frac{\partial}{\partial x^{i}}\), and \(X\) denotes the geodesic vector field. Integral curves of the magnetic flow preserve \(H\), and thus have constant speed. In what follows we will restrict ourselves to working on the unit tangent bundle: \(SM:=H^{-1}(\frac{1}{2})\). This is not a genuine restriction from a dynamical point of view, since other energy levels may be understood by simply changing \(\Omega\) to \(c\Omega\), where \(c\in \mathbb{R}.\) Furthermore, the magnetic geodesics, that is the projection of the integral curves of the magnetic flow to \(M\), are precisely the solutions \(t\mapsto (\gamma(t),\dot{\gamma}(t))\) to the following equation: \[\nabla_{\dot{\gamma}} \dot{\gamma}=Y(\dot{\gamma}).\] For oriented surfaces \(SM\) is an \(S^{1}\)-fibration with a circle action on the fibres inducing a vector field which we shall denote by \(V\). It is a routine exercise to show that for surfaces the generator of the magnetic flow in fact simplifies to: \(X+\lambda V,\) where \(\lambda\in C^{\infty}(M)\) is the unique function satisfying: \(\Omega=\lambda d\mbox{vol}_{g}\) for \(d\mbox{vol}_{g}\) the area form of \(M.\) In this article we will proceed under the assumption that \(M\) is a surface.\\
\indent Given a magnetic flow \(\varphi_{t}\) on \(M\), denote by \(\mathcal{G}(M,g,\Omega)\) the set of periodic orbits. We may define the \textit{magnetic ray transform} of a function \(f\in C^{\infty}(SM)\) by \(I(f):\mathcal{G}(M,g,\Omega)\rightarrow\mathbb{R}\)
\[I(f)(\gamma):=\int_{0}^{T}f(\gamma(t),\dot{\gamma}(t))dt,\ \ \gamma\in\mathcal{G}(M,g,\Omega)\ \mbox{has period }T.\]
Now one is lead to ask about the injectivity of this mapping. That is, assume \(I(f)(\gamma)=0\) for all closed orbits \(\gamma\) of the magnetic flow, is \(f\equiv 0\)? Clearly, if one takes \(f=(X+\lambda V)u\) for \(u\in C^{\infty}(SM)\), then \(I(f)\equiv 0.\) Hence, there is a natural obstruction to injectivity, however, the question remains: are these the only elements in the kernel? In order to characterize the kernel of the ray transform so succintly one would expect to have to impose some condition on the flow itself, so that the space of closed orbits is sufficiently rich. To this end we stipulate that our flow is Anosov. This means that there exists a continuous splitting \(T(SM)=E^{0}\oplus E^{u}\oplus E^{s}\) where \(E^{0}\) is the flow direction, and there are constants \(C>0\) and \(0<\rho<1<\eta\) such that for all \(t>0\) we have
\[\left\|d\varphi_{-t}|_{E^{u}}\right\|\leq C\eta^{-t}\ \ \mbox{and}\ \ \left\|d\varphi_{t}|_{E^{s}}\right\|\leq C\rho^{t}\]
It is shown in \cite{anosov167} that given a measure preserving Anosov flow, its periodic orbits are dense in the space of all orbits. (Magnetic flows on \(SM\) preserve the Liouville measure, induced from the volume form specified below.) 
The smooth Livsic theorem \cite{llave86} indeed shows that given a transitive, Anosov flow \(I(f)=0\) iff there exists \(u\in C^{\infty}(SM)\) such that \((X+\lambda V)u=f\). (Since the magnetic flow is volume preserving, its non-wandering set is all of \(SM\), therefore if the flow is, in addition, Anosov, then it must be transitive.) Now we wish to ask a more refined question about the kernel. In order to do so we need to digress to introduce some Fourier analysis.\\
\indent As previously let \(X\) denote the vector field on \(SM\) generated by the geodesic flow, and \(V\) the vector field induced by the circle action on the fibres. We define a third vector field as follows: \(X_{\bot}:=[X,V]\). The two remaining commutators will play an important role in what follows, and are given by (see \cite{singer67}): \([V,X_{\bot}]=X\), and \([X,X_{\bot}]=-KV\), where \(K\) is the Gaussian curvature of \(M\). A global frame for \(T(SM)\) is given by \(\left\{ X+\lambda V,X_{\bot},V \right\}\). We define a Riemannian metric on \(SM\) by declaring that \(\left\{X+\lambda V,X_{\bot},V\right\}\) form an orthonormal basis, and will denote by \(d\Sigma^{3}\) the volume form of this metric. Note that this volume form is identical to the one obtained by declaring that \(\left\{X,X_{\bot},V\right\}\) form an orthonormal basis.\\
\indent We define an inner product between functions \(u,v:SM\rightarrow \mathbb{C}\) as follows: 
\[\left\langle u,v\right\rangle:=\int_{SM}\left\langle u,v\right\rangle _{\mathbb{C}} d\Sigma^{3}.\]
The \(H^{1}(SM)\) norm of \(u\in C^{\infty}(SM)\) will be defined as: 
\[\left\|u\right\|_{H^{1}(SM)}:=\left\|(X+\lambda V)u\right\|_{L^{2}(SM)}+\left\|X_{\bot}u\right\|_{L^{2}(SM)}+\left\|Vu\right\|_{L^{2}(SM)}+\left\|u\right\|_{L^{2}(SM)}.\]
Note that with respect to the metric on \(SM\) this can be written as \(\left\|u\right\|_{H^{1}(SM)}=\left\|\nabla u\right\|_{L^{2}(SM)}+\left\|u\right\|_{L^{2}(SM)}.\)
Now, \(L^{2}(SM,\mathbb{C})\) decomposes orthogonally as:
\[L^{2}(SM,\mathbb{C})=\bigoplus_{k\in\mathbb{Z}}H_{k}\]
where \(-iV\) acts as \(k\operatorname{Id}\) on \(H_{k}.\) Thus, we can decompose a smooth function \(u:SM\rightarrow\mathbb{C}\) into its Fourier components 
\[u=\sum^{\infty}_{k=-\infty}u_{k}\]
where \(u_{k}\in\Omega_{k}:=C^{\infty}(SM,\mathbb{C})\cap H_{k}.\)\\
\indent Given any symmetric, covariant \(m\)-tensor field \(f=f_{i_{1}\cdots i_{m}}dx^{i_{1}}\otimes\cdots\otimes dx^{i_{m}}\) on \(M\) we can associate a function \(\hat{f}\) on \(SM\) by 
\[\hat{f}(x,v):=f_{i_{1}\cdots i_{m}}v^{i_{1}}\cdots v^{i_{m}}.\]
Now such a \(\hat{f}\) can be decomposed into its Fourier components as \(\hat{f}=\sum_{k=-m}^{m}\hat{f}_{k}\), and in general if any function on \(SM\) has nontrivial Fourier components only for \(-m\leq k\leq m\), then we say that it is of degree \(m\). In what follows we will drop the hat leaving it clear from the context when we mean \(f\) to induce a function on \(SM\).\\
\indent The tensor tomography problem is to determine the kernel of the ray transform when it acts on functions on \(SM\) which are induced by tensor fields. Let us assume at first that \(\Omega=0\), so that we are working with the standard geodesic ray transform. Given our setup the Livsic theorem implies that for any symmetric \(m\)-tensor \(f\) satisfying \(I(f)=0\), there exists \(u\in C^{\infty}(SM)\) such that \(Xu=f\). The right hand side has degree \(m\). Does this imply that \(u\) has degree \(m-1\)? (When \(f\) has degree \(0\) we interpret the question to be: does \(f\equiv 0\)?) This inverse question has received considerable attention recently. It is well known that the geodesic flow of a negatively curved \(n\)-manifold is Anosov (though the converse doesn't hold), moreover in \cite{guillemin80} it is shown that on a closed, oriented, negatively curved surface our question is resolved affirmatively, and furthermore in \cite{croke98} that on a closed, oriented, negatively curved \(n\)-manifold, the same result holds. Ideally, one would like to extend these results to the general Anosov case, and remove the curvature assumption. Partial results in this direction were achieved in \cite{dairbekov03}. In particular, for the Anosov case for surfaces it is shown there that the statement is true for tensors of rank \(m=0,1\), and in \cite{sharafutdinov00} it's shown that it holds for \(m=2\) with the additional assumption that \((M,g)\) has no focal points. Finally, in \cite{paternain12} the case for \(m=2\) has been recently resolved.\\
\indent Now let us proceed to the general magnetic case where \(\Omega\) is arbitrary. In the magnetic setting, the flow couples components of differing degrees, hence the analogous question requires us to consider sums of functions induced by tensors of differing ranks. Therefore, for each \(0\leq i\leq m\) let \(f_{i}\) be a symmetric \(i\)-tensor, inducing a function on \(SM\), and consider \(f=f_{0}+...+f_{m}\). Given our setup the Livsic theorem implies that if \(f\) satisfies \(I(f)=0\), there exists \(u\in C^{\infty}(SM)\) such that \((X+\lambda V)u=f\). The right hand side has degree \(m\). Does this imply that \(u\) has degree \(m-1\)? (When \(f\) has degree \(0\) we interpret the question to be: does \(f\equiv 0\)?)\\
\indent In the general magnetic case most of the results heretofore are only for tensors up to degree \(1\). In \cite{dairbekov05} the above question for surfaces with an Anosov magnetic flow is resolved affimatively for tensors up to degree \(1\). In \cite{dairbekov207} the same statement is proved, but the more general class of Anosov thermostat flows are considered, and in \cite{jane08} some further partial results for thermostats are achieved. In \cite{dairbekov08} it is shown that the Riemannian hypothesis can be weakened to Finsler, and the statement proven for tensors up to degree \(1\) on manifolds of arbitrary dimension. In \cite{dairbekov108} positive results are obtained even when the flow is not Anosov, but simply has no conjugate points.\\
\indent This leads us to the main result of this paper

\begin{theorem} Let \((M,g,\Omega)\) be a closed, oriented Riemannian surface equipped with an Anosov magnetic flow. Let \(f=f_{0}+f_{1}+f_{2}\) where \(f_{i}\) is a symmetric \(i\)-tensor. If \(I(f)=0\), then \(f=(X+\lambda V)a\) where \(a\in C^{\infty}(SM)\) is of degree \(1\).
\label{anosovi2}
\end{theorem}

When dealing with magnetic flows negative Gaussian curvature alone is not sufficient to guarantee that the flow is Anosov. It turns out that the appropriate quantity to consider in the magnetic setting is magnetic curvature, which we follow \cite{burns02} in defining to be \(\mathbb{K}:=K+X_{\bot}(\lambda)+\lambda^{2}\). Now if \((M,g,\Omega)\) has negative magnetic curvature, then the magnetic flow is Anosov \cite{wojtkowski08}, in analogy with the geodesic case. Thus, our theorem above yields the following result, which answers a question initially posed in \cite{jane08}.

\begin{corollary} Let \((M,g,\Omega)\) be a closed, oriented Riemannian surface equipped with a magnetic flow. Suppose that \(\mathbb{K}<0\). Let \(f=f_{0}+f_{1}+f_{2}\) where \(f_{i}\) is a symmetric \(i\)-tensor. If \(I(f)=0\), then \(f=(X+\lambda V)a\) where \(a\in C^{\infty}(SM)\) is of degree \(1\). \label{mainresult}
\end{corollary}

\begin{remark} In the above theorem and corollary we used \(f_{i}\) to denote a symmetric \(i\)-tensor. This is the only instance where we will use this notation for this purpose - in the remainder of this article we will only ever use \(f_{i}\) to denote the \(i^{th}\) Fourier component of \(f\), a function on \(SM\) (possibly induced by a tensor).
\end{remark}

Concerning the proof of Theorem~\ref{anosovi2} we need only consider the case when the genus of \(M\) is \(\geq 2\). This is a consequence of the fact that the fundamental group of any \(S^{1}\)-bundle over the \(2\)-sphere or torus has polynomial growth, and a classic result of Plante and Thurston \cite{plante72} which says that if an \(S^{1}\)-fibration over a surface carries an Anosov flow, then the fundamental group of the fibration must grow exponentially.\\ 
\indent We will exploit the following lemma to achieve our result. Its proof is given in \cite{paternain12}.

\begin{lemma} Assume the Riemannian surface \((M,g)\) has genus \(\geq 2\). Then \(\eta_{+}:\Omega_{k}\rightarrow \Omega_{k+1}\) is injective for \(k\geq 1\), and \(\eta_{-}:\Omega_{k}\rightarrow \Omega_{k-1}\) is injective for \(k\leq -1\). Here \(\eta_{+}:=\frac{1}{2}(X+iX_{\bot})\) and \(\eta_{-}:=\frac{1}{2}(X-iX_{\bot})\).
\label{injectivity}
\end{lemma}

\indent Another crucial component in our proof is the Pestov Identity (see \cite{dairbekov207} for a succinct proof) which has recently been used in various guises in the resolution of inverse problems, see for example \cite{dairbekov08, paternain111, paternain211}.

\begin{theorem}[\sc{Pestov's Identity}] Let \((M,g,\Omega)\) be a closed, oriented Riemannian surface equipped with a magnetic flow. If \(u\in C^{\infty}(SM)\) then the following holds:
\[\left\|V((X+\lambda V)u)\right\|^{2}=\left\|(X+\lambda V)Vu\right\|^{2}-(\mathbb{K}Vu,Vu)+\left\|(X+\lambda V)u\right\|^{2}.\]
\end{theorem}

\indent In the case where one has non-empty boundary the geodesic tensor tomography question has been fully resolved for \textit{simple} surfaces by Paternain, Salo and Uhlmann in \cite{paternain211}, where a compact Riemannian surface with non-empty boundary is said to be simple if its boundary is strictly convex and given any point \(p\in M\) the exponential map \(\operatorname{exp}_{p}\) is a diffeomorphism onto \(M\). The definition was originally motivated by the boundary rigidity problem \cite{michel80}. The magnetic tensor tomography problem was resolved by extending these techniques to \textit{simple magnetic systems} in \cite{ainsworth12}. We say \((M,g,\Omega)\) is a simple magnetic system if (in analogy with a simple surface) its boundary is strictly magnetic convex and given any point \(p\in M\) the magnetic exponential map \(\operatorname{exp}^{\Omega}_{p}\) is a diffeomorphism onto \(M\) - see \cite{dairbekov107} for further details. As is explained in \cite{paternain211}, in order to get injectivity of \(I_{k}\) one need only assume that \(I_{0}\) and \(I_{1}\) are injective, and that \(I_{0}^{*}\) is surjective - here by \(I_{k}\) we mean the ray transform restricted to smooth functions on \(SM\) induced by symmetric \(k\)-tensors. (It is shown in \cite{mukhometov77} and \cite{anikonov97} respectively that \(I_{0}\) and \(I_{1}\) are injective on simple surfaces, and in \cite{pestov05} it is shown that \(I_{0}^{*}\) is surjective on simple surfaces.) This leads one to investigate the surjectivity of the adjoint of the ray transform in the closed surface case. One major discrepancy with the boundary case is that in the closed case there exist no non-trivial solutions to: \((X+\lambda V)u=0\) which lie in \(L^{2}(SM)\). This is because an Anosov, volume preserving flow is necessarily ergodic, and so the only \(u\in L^{2}(SM)\) which are invariant under the flow are the constants. This is in contrast with the boundary case \cite{pestov05}. Hence, the optimal integral regularity for solutions to \((X+\lambda V)u=0\) is \(H^{-1}(SM)\). For this reason we pause to introduce distributions.\\
\indent Denote by \(\mathcal{D}'(SM)\) the space of distributions, or continuous linear functionals, on \(C^{\infty}(SM)\). Any vector field \(W\) can act on a distribution \(\mu\in\mathcal{D}'(SM)\), via duality. That is, \((W\mu)\varphi:=\mu(W\varphi)\) for any \(\varphi\in C^{\infty}(SM)\). This leads us to define the space of distributions invariant under the magnetic flow:
\[\mathcal{D}'_{\operatorname{inv}}(SM):=\left\{\mu\in\mathcal{D}'(SM):(X+\lambda V)\mu=0\right\}.\]
It is explained in \cite{paternain12} how we can without loss of generality consider the magnetic ray transform as the following map:
\[I:C^{\infty}(SM)\rightarrow L(\mathcal{D}'_{\operatorname{inv}}(SM),\mathbb{R}),\ \ If(\mu):=\mu(f).\]
The adjoint of \(I\) is given by the map:
\[I^{*}:\mathcal{D}'_{\operatorname{inv}}(SM)\rightarrow \mathcal{D}'(SM),\ \ (I^{*}\mu)\varphi=\mu(I\varphi)\ \ \forall\varphi\in C^{\infty}(SM).\]  
Note that we can use duality to decompose a distribution into its Fourier components in the same way that we decompose a function into its Fourier components. That is, \(\mu_{k}(\varphi):=\mu(\varphi_{k}),\ \ \forall\varphi\in C^{\infty}(SM)\). With this preparation we may now proceed to state two results on the surjectivity of \(I^{*}\). Firstly, we have surjectivity for \(I_{0}^{*}\).

\begin{theorem} Let \((M,g,\Omega)\) be a closed, oriented Riemannian surface equipped with an Anosov magnetic flow. If \(f\in C^{\infty}(M)\), then there exists \(w\in H^{-1}(SM)\) with \((X+\lambda V)w=0\) and \(w_{0}=f\).
\label{theorem14}
\end{theorem}    

The next theorem relates to the surjectivity of the adjoint of the ray transform restricted to functions on \(SM\) induced by \(1\)-forms. To give some intuition behind the technical condition, in \cite{paternain12} it is shown that \(a_{-1}+a_{1}\in\Omega_{-1}\oplus\Omega_{1}\) satisfies \(\eta_{+}a_{-1}+\eta_{-}a_{1}=0\) iff the \(1\)-form \(\sigma\) associated with \(a_{-1}+a_{1}\) is solenoidal, in the sense that \(\delta \sigma =0\).  

\begin{theorem} Let \((M,g,\Omega)\) be a closed, oriented Riemannian surface equipped with an Anosov magnetic flow. If \(a_{-1}+a_{1}\in\Omega_{-1}\oplus\Omega_{1}\) satisfies \(\eta_{+}a_{-1}+\eta_{-}a_{1}=0\), then there exists \(w\in H^{-1}(SM)\) such that \((X+\lambda V)w=0\) and \(w_{-1}+w_{1}=a_{-1}+a_{1}\).
\label{theorem55}
\end{theorem}

Using the ideas from \cite{paternain12} in Section~\ref{sec:inj} we make judicious use of Theorem~\ref{theorem55} to resolve the tensor tomography problem for sums of tensors of degree at most \(2\).\\
\indent The final result in this article is to give an application of our methods to explore a question related to the entropy production of magnetic Anosov flows perturbed by symmetric \(2\)-tensors. In order to state our theorem we introduce the following concepts \cite{ruelle2}:

\begin{definition} Given a smooth manifold \(N\) equipped with a flow \(\psi_{t}\) with generator \(\mathbf{G}\). We say that a \(\psi_{t}\)-invariant measure \(\rho\) is an SRB measure if \(\rho\) is ergodic and 
\[h_{\rho}(\psi_{t}):=\sum\mbox{positive Lyapunov exponents},\]
where \(h_{\rho}(\psi_{t})\) is the measure theoretic entropy of \(\psi_{t}\) with respect to \(\rho\). The entropy production of the measure \(\rho\) is defined to be:
\[e_{\psi_{t}}(\rho):=-\int_{N} \operatorname{div}\mathbf{G}\,d\rho=-\sum\mbox{Lyapunov exponents}.\]
\end{definition} 

\begin{remark} It has been shown in \cite{ruelle2} that \(e_{\psi_{t}}(\rho)\geq 0\) with equality iff
\[h_{\rho}(\psi_{t})=\sum\mbox{positive Lyapunov exponents}=-\sum\mbox{negative Lyapunov exponents}.\]
\end{remark}

\begin{remark} When one has an Anosov flow on a surface a result of Ghys \cite{ghys84} implies it is transitive and topologically mixing, moreover, for such a flow the SRB measure is known to be unique \cite{katok}.
\end{remark}

\indent  In our situation we have \((M,g,\Omega)\) an oriented Riemannian surface equipped with an Anosov magnetic flow on \(SM\). Now suppose \(q\) is a fixed symmetric covariant \(k\)-tensor on \(M\). We use \(q\) to determine another flow on \(SM\), denoted by \(\varphi^{sq}_{t}\) with generator: 
\[\mathbf{G}_{s}:=X+(\lambda +sq)V.\]
For \(\varepsilon\) sufficiently small and \(s\in(-\varepsilon,\varepsilon)\), then \(\varphi^{sq}_{t}\) will remain a transitive, weakly mixing Anosov flow by structural stability. We refer to this as the magnetic flow perturbed by \(q\). \\
\indent Finally, note that for all \(s\in(-\varepsilon,\varepsilon)\) the map:
\[e(s):=e_{\varphi^{sq}_{t}}(\rho_{s})\]
is smooth by the results of \cite{ruelle3}. Here \(\rho_{s}\) is the unique SRB measure. This preamble allows us to state our theorem.

\begin{theorem} Let \((M,g,\Omega)\) be a closed, oriented Riemannian surface equipped with an Anosov magnetic flow which we shall denote by \(\varphi_{t}\). Fix a symmetric covariant \(2\)-tensor \(q\). The perturbed flow \(\varphi^{sq}_{t}\) has zero entropy production iff \(V(q)\) is purely potential, that is, iff there exists \(u\in C^{\infty}(SM)\) of degree \(1\), such that \((X+\lambda V)u=V(q)\).
\label{entropy}
\end{theorem}

Throughout this paper we are restricting ourselves to the case where \(M\) is closed, however, as has been mentioned, many analogous questions have been asked when \(M\) is compact with non-empty boundary, see for example \cite{ainsworth12, dairbekov107, paternain111, paternain211, pestov05}.

\section{\(\alpha\)-controlled Estimates and Surjectivity of \(I^{*}\)}
\label{sec:surj}

\begin{definition}
Let \(\alpha\in[0,1]\). We say that \((M,g,\Omega)\) is \(\alpha\)-controlled if 
\[\left\|(X+\lambda V)u\right\|^{2}-(\mathbb{K}u,u)\geq \alpha \left\|(X+\lambda V)u\right\|^{2}\]
for all \(u \in C^{\infty}(SM)\).
\end{definition}
In the proofs that follow we will denote by \(T:C^{\infty}(SM)\rightarrow \bigoplus_{\left|k\right|\geq m+1}\Omega_{k}\) the projection operator, defined by: 
\[Tu=\sum_{\left|k\right|\geq m+1}u_{k}.\]
In addition we define \(Q:C^{\infty}(SM)\rightarrow \bigoplus_{\left|k\right|\geq m+1}\Omega_{k}\) as follows \(Qu:=TV((X+\lambda V)u)\).

The following proposition proved crucial in achieving the surjectivity of \(I^{*}\) in the geodesic setting. We include it here to exhibit the difficulties associated with adapting such techniques to the magnetic setting due to the presence of the extra terms resulting from the coupled equations.

\begin{proposition} Let \((M,g,\Omega)\) be a closed, oriented Riemannian surface equipped with a magnetic flow. Suppose \((M,g,\Omega)\) is \(\alpha\)-controlled and let \(m\) be an integer \(\geq 2\). Then given any \(u\in\bigoplus_{\left|k\right|\geq m}\Omega_{k}\) we have 
\begin{align}
\left\|Qu\right\|^{2} \geq& \left\|v\right\|^{2} + \alpha \left\|w\right\|^{2} +\left(1-(m-1)^{2}+\alpha m^{2}\right)\left(\left\|\eta_{-}u_{m}\right\|^{2}+\left\|\eta_{+}u_{-m}\right\|^{2}\right) \nonumber\\
                      & + \left(1-m^{2}\right)\left(\left\|im\lambda u_{m}+\eta_{-}u_{m+1}\right\|^{2}+\left\|-im\lambda u_{-m}+\eta_{+}u_{-(m+1)}\right\|^{2}\right) \nonumber\\
                      & + \alpha \left(\left\|-m^{2}\lambda u_{m}+\eta_{-}i(m+1)u_{m+1}\right\|^{2}+\left\|-m^{2}\lambda u_{-m}-\eta_{+}i(m+1)u_{-(m+1)}\right\|^{2} \right) \nonumber
\end{align}
where \(v:=\sum_{\left|k\right|\geq m+1}((X+\lambda V)u)_{k}\) and \(w:=\sum_{\left|k\right|\geq m+1}((X+\lambda V)Vu)_{k}\). 
\label{inequality}
\end{proposition}

\begin{proof} Given \(u\in\bigoplus_{\left|k\right|\geq m}\Omega_{k}\) we compute:
\begin{align*}
\sum_{\left|k\right|\leq m}k^{2}\left\|((X+\lambda V)u)_{k}\right\|^{2}  = &(m-1)^{2}\left\|\eta_{-}u_{m}\right\|^{2}+m^{2}\left\|im\lambda u_{m}+\eta_{-}u_{m+1}\right\|^{2}+(m-1)^{2}\left\|\eta_{+}u_{-m}\right\|^{2}\\
                                                                        & +m^{2}\left\|-im\lambda u_{-m}+\eta_{+}u_{-(m+1)}\right\|^{2}.\\                                                           \\ \left\|(X+\lambda V)u\right\|^{2}  = &\left\|im\lambda u_{m}+\eta_{-}u_{m+1}\right\|^{2}+\left\|\eta_{-}u_{m}\right\|^{2}+\left\|i(-m)\lambda u_{-m}+\eta_{+}u_{-(m+1)}\right\|^{2}+\left\|\eta_{+}u_{-m}\right\|^{2}\\
                                                                        & +\left\|\sum_{\left|k\right|\geq m+1}((X+\lambda V)u)_{k}\right\|^{2}.\\
                                  \\ \left\|(X+\lambda V)Vu\right\|^{2}  = &\left\|-m^{2}\lambda u_{m}+\eta_{-}i(m+1)u_{m+1}\right\|^{2}+\left\|\eta_{-}mu_{m}\right\|^{2}+\left\|-m^{2}\lambda u_{-m}-\eta_{+}i(m+1)u_{-(m+1)}\right\|^{2}\\
                                                                        & +\left\|\eta_{+}mu_{-m}\right\|^{2}+\left\|\sum_{\left|k\right|\geq m+1}((X+\lambda V)Vu)_{k}\right\|^{2}.
\end{align*}

\begin{align}
                                     \left\|V((X+\lambda V)u)\right\|^{2}  = &\sum_{\left|k\right|\leq m}k^{2}\left\|((X+\lambda V)u)_{k}\right\|^{2}+\left\|TV((X+\lambda V)u)\right\|^{2}\nonumber \\
                                                                         = &(m-1)^{2}\left\|\eta_{-}u_{m}\right\|^{2}+m^{2}\left\|im\lambda u_{m}+\eta_{-}u_{m+1}\right\|^{2} + (m-1)^{2}\left\|\eta_{+}u_{-m}\right\|^{2}\nonumber\\
                                                                        & +m^{2}\left\|-im\lambda u_{-m}+\eta_{+}u_{-(m+1)}\right\|^{2}+\left\|TV((X+\lambda V)u)\right\|^{2}.\label{yes}
\end{align}

Now we use Pestov's Identity and our hypothesis:
\begin{align*}
&\left\|V((X+\lambda V)u)\right\|^{2} \\
   &\ \ \ \ = \left\|(X+\lambda V)Vu\right\|^{2}-(\mathbb{K}Vu,Vu)+\left\|(X+\lambda V)u\right\|^{2}\\
   &\ \ \ \ \geq \alpha \left\|(X+\lambda V)Vu\right\|^{2} + \left\|(X+\lambda V)u\right\|^{2}\\
   &\ \ \ \ =\alpha \left\|-m^{2}\lambda u_{m}+\eta_{-}i(m+1)u_{m+1}\right\|^{2} + \alpha \left\|\eta_{-}mu_{m}\right\|^{2}\\
   &\ \ \ \ \ \ \ \ +\alpha \left\|-m^{2}\lambda u_{-m}-\eta_{+}i(m+1)u_{-(m+1)}\right\|^{2} + \alpha \left\|\eta_{+}mu_{-m}\right\|^{2} +\alpha \left\|w\right\|^{2}\\
   &\ \ \ \ \ \ \ \ +\left\|im\lambda u_{m}+\eta_{-}u_{m+1}\right\|^{2}+\left\|\eta_{-}u_{m}\right\|^{2}+\left\|i(-m)\lambda u_{-m}+\eta_{+}u_{-(m+1)}\right\|^{2}+\left\|\eta_{+}u_{-m}\right\|^{2}+\left\|v\right\|^{2}.
\end{align*}
To conclude we simply combine this inequality with equation (\ref{yes}).

\end{proof}

\begin{theorem} Let \((M,g,\Omega)\) be a closed, oriented Riemannian surface equipped with an Anosov magnetic flow. There exists an \(\alpha>0\) such that the following inequality holds for all \(\psi\in C^{\infty}(SM)\):
\[\left\|(X+\lambda V)\psi\right\|^{2}-(\mathbb{K}\psi,\psi)\geq \alpha\left(\left\|(X+\lambda V)\psi\right\|^{2}+\left\|\psi\right\|^{2}\right)\]
In particular, \((M,g,\Omega)\) is \(\alpha\)-controlled.
\label{theorem32}
\end{theorem}
\begin{proof} In \cite{dairbekov207} it is shown that there exist two continuous real-valued functions \(r^{\pm}\) on \(SM\) which are differentiable along the magnetic flow, and which satisfy the following Riccati type equation:
\[(X+\lambda V)r+r^{2}+\mathbb{K}=0.\]
Moreover, it is shown that \(r^{+}-r^{-}>0\). 
Consider,
\begin{align*}
\left|(X+\lambda V)\psi-r\psi\right|^{2} =& \left|(X+\lambda V)\psi\right|^{2}-2\operatorname{Re}\left\{r((X+\lambda V)\psi)\overline{\psi}\right\}+r^{2}\left|\psi\right|^{2}\\
                                         =& \left|(X+\lambda V)\psi\right|^{2}+\left|\psi\right|^{2}((X+\lambda V)r+r^{2})-(X+\lambda V)(r\left|\psi\right|^{2}).
\end{align*} 
One now integrates this over \(SM\), and uses both the Riccati equation and the fact that the volume form is invariant under the magnetic flow to obtain:
\[\left\|(X+\lambda V)\psi-r\psi\right\|^{2}=\left\|(X+\lambda V)\psi\right\|^{2}-(\mathbb{K}\psi,\psi).\]
Defining \(A:=(X+\lambda V)\psi-r^{-}\psi\) and \(B:=(X+\lambda V)\psi-r^{+}\psi\) we note that the previous equation guarantees that \(\left\|A\right\|=\left\|B\right\|\), and we can solve for \(\psi\) and \((X+\lambda V)\psi\) to get:
\begin{align*}
\psi =& (r^{+}-r^{-})^{-1}(A-B)\\
(X+\lambda V)\psi =& cA+(1-c)B,
\end{align*}
where \(c:=r^{+}/(r^{+}-r^{-})\). Therefore we can choose an \(\alpha>0\) such that:
\begin{align*}
2\alpha\left\|\psi\right\|^{2}\leq &\left\|A\right\|^{2}\\
2\alpha\left\|(X+\lambda V)\psi\right\|^{2}\leq &\left\|A\right\|^{2}.
\end{align*}
Thus yielding the desired inequality.
\end{proof}

We introduce the following operator \(P:C^{\infty}(SM)\rightarrow C^{\infty}(SM),\ \ Pu:=V((X+\lambda V)u)\). In addition, given \(E\) a subspace of \(\mathcal{D}'(SM)\) we define \(E_{\diamond}:=\left\{\mu\in E: \left\langle \mu,1\right\rangle=0.\right\}\)

\begin{lemma} Let \((M,g,\Omega)\) be a closed, oriented Riemannian surface equipped with an Anosov magnetic flow. Then there exists a constant \(C>0\) such that
\[\left\|u\right\|_{H^{1}(SM)}\leq C\left\|Pu\right\|_{L^{2}(SM)},\ \ \forall u\in C^{\infty}_{\diamond}(SM).\]
\label{lem:42}
\end{lemma}
\begin{proof} To begin we combine the Pestov Identity with Theorem~\ref{theorem32} to obtain for any \(u\in C^{\infty}(SM)\):
\begin{align*}
\left\|V((X+\lambda V)u)\right\|^{2} =& \left\|(X+\lambda V)Vu\right\|^{2}-(\mathbb{K}Vu,Vu)+\left\|(X+\lambda V)u\right\|^{2}\\
                                     \geq& \left\|(X+\lambda V)u\right\|^{2}+\alpha (\left\|(X+\lambda V)Vu\right\|^{2}+\left\|Vu\right\|^{2}).
\end{align*}
Recall \(X_{\bot}u=[X,V]u=[X+\lambda V,V]u\). Therefore,
\[\left\|X_{\bot}u\right\|^{2}\leq 2(\left\|(X+\lambda V)Vu\right\|^{2}+\left\|V((X+\lambda V)u)\right\|^{2}),\]
and so
\[\left\|(X+\lambda V)Vu\right\|^{2}\geq \frac{1}{2}\left\|X_{\bot}u\right\|^{2}-\left\|V((X+\lambda V)u)\right\|^{2}.\]
Using this we write:
\[\left\|V((X+\lambda V)u)\right\|^{2} \geq \left\|(X+\lambda V)u\right\|^{2} +\alpha\left\|Vu\right\|^{2} + \frac{1}{2}\alpha\left\|X_{\bot}u\right\|^{2}-\alpha\left\|V((X+\lambda V)u)\right\|^{2}.\]
Therefore,
\[\left\|V((X+\lambda V)u)\right\|^{2} \geq \frac{1}{1+\alpha}\left(\left\|(X+\lambda V)u\right\|^{2} +\alpha\left\|Vu\right\|^{2} + \frac{1}{2}\alpha\left\|X_{\bot}u\right\|^{2}\right).\]
By the Poincare Inequality for closed Riemannian manifolds, there exists \(D>0\) such that 
\[\left\|u\right\|^{2}\leq D\left\|\nabla u\right\|=D(\left\|(X+\lambda V)u\right\|^{2}+\left\|X_{\bot}u\right\|^{2}+\left\|Vu\right\|^{2}),\ \ \forall u\in C^{\infty}_{\diamond}(SM).\]
Therefore, there exists \(C>0\) such that
\[\left\|u\right\|_{H^{1}(SM)}\leq C\left\|Pu\right\|,\ \ \forall u\in C^{\infty}_{\diamond}(SM).\]
\end{proof}

\begin{lemma} Let \((M,g,\Omega)\) be a closed, oriented Riemannian surface equipped with an Anosov magnetic flow. For any \(f\in H^{-1}_{\diamond}(SM)\), there exists \(h\in L^{2}(SM)\) satisfying
\[P^{*}h=f\ \ \mbox{in}\ SM.\]
In addition, \(\left\|h\right\|_{L^{2}(SM)}\leq C\left\|f\right\|_{H^{-1}(SM)}\) with \(C>0\) independent of \(f\).
\label{lemma43}
\end{lemma}
\begin{proof} This proof is adapted directly from \cite{paternain12} and is very similar to the proof of Lemma~\ref{lemma54} later in this article, hence we omit it.
\end{proof}

\begin{proof}[Proof of Theorem~\ref{theorem14}] Given \(f\in C^{\infty}(M)\) Lemma~\ref{lemma43} ensures there exists \(h\in L^{2}(SM)\) such that 
\[P^{*}h=-(X+\lambda V)f.\]
Define \(w:= Vh+f\). Then,
\[(X+\lambda V)w=(X+\lambda V)Vh + (X+\lambda V)f=0.\]
Clearly, \(w_{0}=f\).
\end{proof}

The following proposition gives the crucial estimate for establishing the surjectivity of the adjoint of the ray transform.

\begin{proposition} Let \((M,g,\Omega)\) be a closed, oriented Riemannian surface equipped with an Anosov magnetic flow. Then there exists a constant \(C>0\) such that
\[\left\|u\right\|_{H^{1}(SM)}\leq C\left\|Qu\right\|_{L^{2}(SM)}\]
for any \(u\in\bigoplus_{|k|\geq 1} \Omega_{k}\).
\label{prop:new}
\end{proposition}
\begin{proof}
In this proof we will denote by \(C>0\) a constant which will be permitted to change from line to line for ease of notation. Let \(u\in\bigoplus_{|k|\geq 1} \Omega_{k}\). Now from the definition
\[\left\|Pu\right\|^{2} = \sum_{|k|\leq 1}k^{2}\left\|\left(\left(X+\lambda V\right)u\right)_{k}\right\|^{2}+\left\|Qu\right\|^{2}.\]
Moreover, from Lemma~\ref{lem:42} we know that 
\[\left\|u\right\|_{H^{1}(SM)}\leq C\left\|Pu\right\|_{L^{2}(SM)}.\]
Thus, it remains to show that \(\left\|\left((X+\lambda V)u\right)_{\pm 1}\right\|\leq C\left\|Qu\right\|\). Consider that by Theorem~\ref{theorem32} and the Pestov Identity we have
\begin{align}
\left\|Pu\right\|^{2} &= \sum_{|k|\leq 1}k^{2}\left\|\left(\left(X+\lambda V\right)u\right)_{k}\right\|^{2}+\left\|Qu\right\|^{2}\nonumber\\
                      &\geq \left\|(X+\lambda V)u\right\|^{2}+\alpha \left\|(X+\lambda V)Vu\right\|^{2}+\alpha\left\|Vu\right\|^{2}\label{eq:alpha}
\end{align}
Therefore, \(\left\|Qu\right\|^{2} \geq \alpha \left\|Vu\right\|^{2}\geq \alpha \left\|iu_{\pm 1}\right\|^{2}\).
And so, 
\begin{equation}
\left\|\lambda u_{\pm 1}\right\|\leq C\left\|Qu\right\|.\label{anotherestimate}
\end{equation}
Again considering equation~(\ref{eq:alpha}),
\begin{align*}
\left\|Qu\right\|^{2} &\geq \alpha \left\|(X+\lambda V)Vu\right\|^{2}\\
                      &\geq \alpha \left\|\left((X+\lambda V)Vu\right)_{1}\right\|^{2}+\alpha \left\|\left((X+\lambda V)Vu\right)_{-1}\right\|^{2}\\
                      &\geq \alpha\left\|2i\eta_{-}u_{2}-\lambda u_{1}\right\|^{2}+\alpha\left\|-2i\eta_{+}u_{-2}-\lambda u_{-1}\right\|^{2}
\end{align*}
Therefore, \(\left\|2i\eta_{-}u_{2}-\lambda u_{1}\right\|\leq C\left\|Qu\right\|\) and \(\left\|2i\eta_{+}u_{-2}+\lambda u_{-1}\right\|\leq C\left\|Qu\right\|.\)
Using the reverse triangle inequality we obtain 
\[\left\|2i\eta_{-}u_{2}\right\|-\left\|\lambda u_{1}\right\|\leq C\left\|Qu\right\|\ \ \mbox{and}\ \ \left\|2i\eta_{+}u_{-2}\right\|-\left\|\lambda u_{-1}\right\|\leq C\left\|Qu\right\|\]
Combining this with our previous estimate (\ref{anotherestimate}), gives 
\[\left\|\eta_{-}u_{2}\right\|\leq C\left\|Qu\right\|\ \ \mbox{and}\ \ \left\|\eta_{+}u_{-2}\right\|\leq\left\|Qu\right\|\]
Using these estimates with the triangle inequality again gives
\begin{align*}
\left\|\left((X+\lambda V)u\right)_{1}\right\| &=  \left\|\eta_{-}u_{2}+i\lambda u_{1}\right\|\\
                                               &\leq  \left\|\eta_{-}u_{2}\right\|+\left\|\lambda u_{1}\right\|\\
                                               &\leq  C\left\|Qu\right\|
\end{align*}
\begin{align*}
\left\|\left((X+\lambda V)u\right)_{-1}\right\| &=  \left\|\eta_{+}u_{-2}-i\lambda u_{-1}\right\|\\
                                                &\leq  \left\|\eta_{+}u_{-2}\right\|+\left\|\lambda u_{-1}\right\|\\
                                                &\leq  C\left\|Qu\right\|
\end{align*}
Thus concluding the proof.
\end{proof}

\begin{lemma} Let \((M,g,\Omega)\) be a closed, oriented Riemannian surface equipped with an Anosov magnetic flow. For any \(f\in H^{-1}(SM)\) with \(f_{0}=0\), there exists \(h\in L^{2}(SM)\) such that
\[Q^{*}h=f.\]
In addition, \(\left\|h\right\|_{L^{2}(SM)}\leq C\left\|f\right\|_{H^{-1}(SM)}\) with \(C>0\) independent of \(f\). 
\label{lemma54}
\end{lemma}
\begin{proof} This proof is adapted directly from Lemma 5.4 given in \cite{paternain12} and is included here for completeness. We define a mapping from the subspace \(Q\left(\bigoplus_{|k|\geq 1}\Omega_{k}\right)\subset L^{2}(SM)\) as follows
\[l:Q\left(\bigoplus_{|k|\geq 1}\Omega_{k}\right)\rightarrow \mathbb{C},\ \ l(Qu):=\left\langle u,f\right\rangle.\]
This is well-defined since any element in \(Q\left(\bigoplus_{|k|\geq 1}\right)\Omega_{k}\) can be written as the image of a unique \(u\in \bigoplus_{|k|\geq 1}\Omega_{k}\) thanks to Proposition~\ref{prop:new}. In addition we have the following inequality
\[|l(Qu)|\leq\left\|f\right\|_{H^{-1}(SM)}\left\|u\right\|_{H^{1}(SM)}\leq C\left\|f\right\|_{H^{-1}(SM)}\left\|Qu\right\|_{L^{2}(SM)}.\]
Hence, \(l\) is continuous, and by the Hahn-Banach Theorem has a continuous extension
\[\overline{l}:L^{2}(SM)\rightarrow\mathbb{C}\ \ \mbox{satisfying}\ \ \left|\overline{l}(v)\right|\leq C\left\|f\right\|_{H^{-1}(SM)}\left\|v\right\|_{L^{2}(SM)}.\]
Now by the Riesz Representation Theorem there exists \(h\in L^{2}(SM)\) such that
\[\overline{l}(v) = \left\langle v,h\right\rangle_{L^{2}(SM)},\ \ \left\|h\right\|_{L^{2}(SM)}\leq C\left\|f\right\|_{H^{-1}(SM)}.\]
If \(u\in C^{\infty}(SM)\), then 
\begin{align*}
\left\langle u,Q^{*}h\right\rangle = &\left\langle Qu,h\right\rangle \\
                                   = &\left\langle Q(u-u_{0}),h\right\rangle \\
                                   = &~l(Q(u-u_{0})) \\
                                   = &\left\langle u-u_{0},f\right\rangle \\
                                   = &\left\langle u,f\right\rangle
\end{align*}
\end{proof}

\begin{proof}[Proof of Theorem~\ref{theorem55}] By Theorem~\ref{theorem32} there exists \(\alpha>0\) such that \((M,g,\Omega)\) is \(\alpha\)-controlled. Define \(f:=-(X+\lambda V)(a_{-1}+a_{1})\), and note that \(f_{0}=0\) by hypothesis. Hence, applying Lemma~\ref{lemma54} shows there exists \(h\in L^{2}(SM)\) such that 
\[Q^{*}h=(X+\lambda V)VTh=-(X+\lambda V)(a_{-1}+a_{1}).\]
It is clear that the distribution \(w:=VTh+a_{-1}+a_{1}\) satisfies the required properties.
\end{proof}

\section{Injectivity for Tensors of Degree at most \(2\)}
\label{sec:inj}

The following theorem will be used to achieve the injectivity of \(I_{2}\), and involves certain mixed norm spaces which we define as:
\[L^{2}_{x}H^{s}_{\theta}(SM):=\left\{u\in\mathcal{D}'(SM):\left\|u\right\|_{L^{2}_{x}H^{s}_{\theta}}<\infty\right\},\ \ \left\|u\right\|_{L^{2}_{x}H^{s}_{\theta}}:=\left(\sum^{\infty}_{k=-\infty}(1+k^{2})^{s}\left\|u_{k}\right\|^{2}_{L^{2}(SM)}\right)^{1/2}\]

\begin{theorem} Let \((M,g,\Omega)\) be a closed, oriented Riemannian surface equipped with an Anosov magnetic flow. If \(a_{1}\in\Omega_{1}\) is such that \(\eta_{-}a_{1}=0\), then there exists \(w=\sum^{\infty}_{k=1}w_{k}\in L^{2}_{x}H_{\theta}^{-1}(SM)\) such that \((X+\lambda V)w=0\), \(w_{1}=a_{1}\), and 
\[\left\|w\right\|_{L^{2}_{x}H_{\theta}^{-1}(SM)}\leq C\left\|a_{1}\right\|_{L^{2}(SM)}\]
for some \(C>0.\)\label{mixednormtheorem}
\end{theorem}
\begin{proof} Define \(a_{-1}:=0\) and designate the distribution given by Theorem~\ref{theorem55} by \(\tilde{w}\). Now project this distribution onto its positive Fourier components to get \(w:=\sum^{\infty}_{k=1}\tilde{w}_{k}\). One can check that \(((X+\lambda V)w)_{k}=0\) for all \(k\), and thus \((X+\lambda V)w=0\). This computation uses the fact that \(\tilde{w}_{0}=0\), which follows because \(\tilde{w}\) has the particular form given in Theorem~\ref{theorem55}. Now \(w=V(\sum^{\infty}_{k=2}h_{k})+a_{1}\) and since \(\left\|h\right\|_{L^{2}(SM)}\leq C\left\|a_{1}\right\|_{L^{2}(SM)}\) we have the desired estimate.
\end{proof}

\begin{theorem} Let \((M,g)\) be a closed, oriented Riemannian surface. Given \(u,v\) distributions in \(SM\) of the form: \(u=\sum^{\infty}_{k=0}u_{k}\), \(v=\sum^{\infty}_{k=0}v_{k}\), where \(u\in L^{2}_{x}H_{\theta}^{-s}\), \(v\in L^{2}_{x}H_{\theta}^{-t}\) for some \(s,t\geq 0\). Define
\[w_{k}:=\sum^{k}_{j=0}u_{j}v_{k-j},\ \ k\in\mathbb{N}.\]
If \(N\in\mathbb{N}\) satisfies \(N>s+t+1/2\), then the sum \(\sum_{k=0}^{\infty}w_{k}\) converges in \(H^{-N-2}(SM)\) to some \(w\) with \(\left\|w\right\|_{H^{-N-2}}\leq C\left\|u\right\|_{L^{2}_{x}H_{\theta}^{-s}}\left\|v\right\|_{L^{2}_{x}H_{\theta}^{-t}}\). Furthermore, 
\[\left\|w_{k}\right\|_{L^{1}(SM)}\leq \left\langle k\right\rangle^{s+t}\left\|u\right\|_{L^{2}_{x}H_{\theta}^{-s}}\left\|v\right\|_{L^{2}_{x}H_{\theta}^{-t}}.\] If \((X+\lambda V)u=(X+\lambda V)v=0\), then \((X+\lambda V)w=0\). \label{distributionestimates}
\end{theorem}  
\begin{proof} This result is proven in \cite{paternain12}. The only additional work required is a computation to verify the last statement. Obviously, \(((X+\lambda V)w)_{k}=0,\ \ \forall k\leq -2\). In addition,
\begin{align*}
\left(\left(X+\lambda V\right)w\right)_{-1}&= \eta_{-}w_{0}=(\eta_{-}u_{0})v_{0}+u_{0}(\eta_{-}v_{0})=0,\\
\left(\left(X+\lambda V\right)w\right)_{0} &= \eta_{-}w_{1}\\
                                           &= (\eta_{-}u_{1})v_{0}+u_{1}(\eta_{-}v_{0})+(\eta_{-}u_{0})v_{1}+u_{0}(\eta_{-}v_{1})=0.
\end{align*}
Furthermore, if \(k\geq 0\), then
\begin{align*}
\left(\left(X+\lambda V\right)w\right)_{k+1} = & \left(Xw\right)_{k+1}+\lambda V(w_{k+1})\\
                                             = &~ \eta_{+}w_{k}+\eta_{-}w_{k+2}+\lambda V(w_{k+1})\\
                                             = & \sum^{k}_{j=0}\left((\eta_{+}u_{j})v_{k-j}+u_{j}(\eta_{+}v_{k-j})\right)+\sum^{k+2}_{j=0}\left((\eta_{-}u_{j})v_{k+2-j}+u_{j}(\eta_{-}v_{k+2-j})\right)+\lambda \sum^{k+1}_{j=0} V\left(u_{j}v_{k+1-j}\right)\\
                                             = & \sum^{k}_{j=0}\left((\eta_{+}u_{j})v_{k-j}+u_{j}(\eta_{+}v_{k-j})\right)+\sum^{k+2}_{j=0}\left((\eta_{-}u_{j})v_{k+2-j}+u_{j}(\eta_{-}v_{k+2-j})\right)\\
                                               & + \sum^{k+1}_{j=0} \left( ij\lambda u_{j}v_{k+1-j} + i(k+1-j)\lambda u_{j}v_{k+1-j} \right)\\
                                             = & \sum^{k}_{j=0}(\eta_{+}u_{j})v_{k-j} + \sum^{k}_{j=0}u_{j}(\eta_{+}v_{k-j}) + \sum^{k+2}_{j=2}(\eta_{-}u_{j})v_{k+2-j} + \sum^{k}_{j=0}u_{j}(\eta_{-}v_{k+2-j})\\
                                               & + \sum^{k+1}_{j=1} ij\lambda u_{j}v_{k+1-j} + \sum^{k}_{j=0} i(k+1-j)\lambda u_{j}v_{k+1-j} \\
                                             = & \sum^{k}_{j=0}(\eta_{+}u_{j})v_{k-j} + \sum^{k}_{j=0}(\eta_{-}u_{j+2})v_{k-j} + \sum^{k}_{j=0} i(j+1)\lambda u_{j+1}v_{k-j} + \sum^{k}_{j=0}u_{j}(\eta_{+}v_{k-j})\\
                                               & + \sum^{k}_{j=0}u_{j}(\eta_{-}v_{k+2-j}) + \sum^{k}_{j=0} i(k+1-j)\lambda u_{j}v_{k+1-j}\\
                                             = & ~0.
\end{align*}
\end{proof}

\begin{proposition} Let \((M,g,\Omega)\) be a closed, oriented Riemannian surface equipped with an Anosov magnetic flow. Suppose \(q\in\Omega_{2}\) is in the linear span of \(\{ab: a,b\in\Omega_{1}\ \eta_{-}(a)=\eta_{-}(b)=0\}\). Then there exists \(w=\sum_{k=2}^{\infty}w_{k}\in H^{-5}(SM)\) such that \((X+\lambda V)w=0\), \(w_{2}=q\), \(\left\|w\right\|_{H^{-5}(SM)}\leq C\left\|q\right\|_{L^{2}(SM)}\).
\label{theorem92}
\end{proposition}
\begin{proof} This follows from Theorem 9.2 in \cite{paternain12} using Theorems~\ref{distributionestimates} \& \ref{mixednormtheorem}.
\end{proof}

Recall the classical fact that a conformal class of Riemannian metrics on a surface determines a complex structure on the manifold. Also, a Riemann surface \(M\) is said to be \textit{hyperelliptic} if there exists a holomorphic map \(f:M\rightarrow S^{2}\) of degree two. Here the degree of a holomorphic map between compact Riemann surfaces is the sum of the multiplicities of the map at every point in the preimage of a fixed point in the codomain \cite{farkas}. When we speak of hyperellipticity of \((M,g)\) below it will always be with respect to the complex structure on \(M\) induced by the metric.

\begin{theorem} Let \((M,g,\Omega)\) be a closed, oriented, non-hyperelliptic Riemannian surface equipped with an Anosov magnetic flow. Suppose \(f\in C^{\infty}(SM)\) is of the form \(f=f_{-2}+f_{-1}+f_{0}+f_{1}+f_{2}\). If there exists \(u\in C^{\infty}(SM)\) such that \((X+\lambda V)u=f\), then \(u\) is of degree \(1\).
\end{theorem}
\begin{proof} Without loss of generality one may assume that both \(f\) and \(u\) are real-valued (otherwise decompose \((X+\lambda V)u=f\) into real and complex components). This ensures that \(\overline{f_{k}}=f_{-k}\) and \(\overline{u_{k}}=u_{-k}\) for all \(k\).\\
\indent Since \(\eta_{+}\) is elliptic \cite{guillemin80} we have the following orthogonal decomposition:
\[\Omega_{2}=\eta_{+}(\Omega_{1})\oplus\mbox{ker}(\eta_{-}).\]
Now, \(f_{2}=\eta_{+}(v_{1})+q_{2}\), where \(q_{2}\in \mbox{ker}(\eta_{-})\) and \(v_{1}\in\Omega_{1}\). Therefore,
\[(X+\lambda V)v_{1}=\eta_{+}(v_{1})+\eta_{-}(v_{1})+i\lambda v_{1}=f_{2}-q_{2}+\eta_{-}(v_{1})+i\lambda v_{1}.\]
Moreover,
\[(X+\lambda V)(u-v_{1})=f_{-2}+f_{-1}+f_{0}+f_{1}+q_{2}-\eta_{-}(v_{1})-i\lambda v_{1}.\]

Since \(M\) is not hyperelliptic, we can invoke Max Noether's Theorem \cite{farkas}, which guarantees that for \(m\geq 2\) the \(m\)-fold products of holomorphic differentials span the space of holomorphic \(m\)-differentials. Using this result with \(m=2\) in conjunction with Lemma~\ref{injectivity} shows that \(q_{2}\) lies in the linear span of the set of products \(a_{1}b_{1}\) where \(a_{1},b_{1}\in\Omega_{1}\) and \(\eta_{-}a_{1}=\eta_{-}b_{1}=0\). Now Propostion~\ref{theorem92} ensures that there exists an invariant distribution \(w=\sum^{\infty}_{k=2}w_{k}\) with \(w_{2}=q_{2}.\) Since \(u-v_{1}\in C^{\infty}(SM)\), we may apply \(w\) to yield:
\[0=\left\langle w, (X+\lambda V)(u-v_{1})\right\rangle=\left\langle w_{2},q_{2}\right\rangle=\left\|q_{2}\right\|^{2}_{L^{2}(SM)}.\]
Thus, \(q_{2}=0\). We also have \(f_{-2}=\overline{f_{2}}=\eta_{-}(\overline{v_{1}}).\)\\
Therefore,
\begin{align*}
(X+\lambda V)(u-v_{1}-\overline{v_{1}}) =& f_{-2}+f_{-1}+f_{0}+f_{1}-\eta_{-}(v_{1})-i\lambda v_{1}-\eta_{+}(\overline{v_{1}})-\eta_{-}(\overline{v_{1}})+i\lambda \overline{v_{1}}\\
                                  =& f_{-1}+f_{0}+f_{1}-\eta_{-}(v_{1})-i\lambda v_{1}-\eta_{+}(\overline{v_{1}})+i\lambda \overline{v_{1}}\in\Omega_{-1}\oplus\Omega_{0}\oplus\Omega_{1}.
\end{align*}
But now thanks to the injectivity of the magnetic ray transform on functions and \(1\)-forms \cite{dairbekov207}, we can deduce that \(u-v_{1}-\overline{v_{1}}\) is of degree \(0\), and thus \(u\) itself is of degree \(1\). 
\end{proof}

\begin{proof}[Proof of Theorem~\ref{anosovi2}] Using the above result this theorem now follows from the argument in \cite{paternain12}.
\end{proof}

\section{Entropy Production of Perturbed Anosov Magnetic Flows}

We first introduce the following concept:
\begin{definition} Suppose that \(\psi_{t}\) is a transitive, weakly mixing Anosov flow on a smooth manifold \(N\) and \(\mu\) is a Gibbs state associated to some H\"{o}lder continuous potential. We define the variance of \(F\), any H\"{o}lder continuous function, with respect to \(\mu\) as:
\[\operatorname{Var}_{\mu}(F):=\lim_{T\rightarrow \infty}\frac{1}{T}\int_{N}\left(\int^{T}_{0}F\circ\psi_{t}-\overline{F}\right)^{2}d\mu,\]
where \[\overline{F}:=\int_{N}Fd\mu.\]
\end{definition}

From \cite{pollicott94} we can equivalently express the variance as 
\[\operatorname{Var}_{\mu}(F)=\int^{\infty}_{-\infty}\sigma_{F}(t)dt=2\int^{\infty}_{0}\sigma_{F}(t)dt,\]
where \[\sigma_{F}(t):=\int_{N}(F\circ\psi_{t}\cdot F-\overline{F}^{2})d\mu.\]
Moreover it has been shown \cite{pollicott85,ruelle1} that the Fourier transform of \(\sigma_{F}\) is defined for \(w\) such that \(\operatorname{Im}(w)>0\), and extends meromorphically near \(w=0\), so we have: 
\[\operatorname{Var}_{\mu}(F)=2\lim_{w\rightarrow 0}\int^{\infty}_{0}e^{iwt}\int_{M}(F\circ\psi_{t}\cdot F-\overline{F}^{2})d\mu.\]

\begin{proof}[Proof of Theorem~\ref{entropy}] A short computation yields 
\[\operatorname{div}(X+(\lambda +sq)V)=sV(q).\] Therefore, by definition we have 
\[ e(s)= -s\int_{SM} V(q)\,d\rho_{s}.\]
Observe that \(\rho_{0}\) is the Liouville measure on \(SM\) and 
\[e'(0)=-\int_{SM} V(q)d\rho_{0}=0.\]
Moreover,
\[e''(0)=-2\left.\frac{d}{ds}\right|_{s=0}\int_{SM} V(q)d\rho_{s}.\]
From the work of \cite{ruelle3} we may compute this last integral. Suppose \(F\in C^{\infty}(SM)\), we have that
\[\left.\frac{d}{ds}\right|_{s=0}\int_{SM} F d\rho_{s}\]
is given by the limit as \(w\rightarrow 0\) (with \(\operatorname{Im}(w)>0\)) of
\[\int^{\infty}_{0}e^{iwt}\int_{SM}d(F\circ\varphi_{t})_{(x,v)}(qV)(x,v)d\rho_{0}(x,v)=-\int^{\infty}_{0}e^{iwt}\int_{SM}\operatorname{div}(qV)(x,v)F(\varphi_{t}(x,v))d\rho_{0}(x,v).\]
Therefore, 
\[e''(0)=2\lim_{w\rightarrow 0}\int^{\infty}_{0}e^{iwt}\int_{SM}V(q)(x,v)V(q)(\varphi_{t}(x,v))d\rho_{0}(x,v),\]
which is none other than the variance of \(V(q)\) (considered as a function on \(SM\)) with respect to the Liouville measure \(\rho_{0}\). That is,
\[e''(0)=\operatorname{Var}_{\rho_{0}}(V(q)).\]
Now \(\operatorname{Var}_{\rho_{0}}(F)\geq 0\) with equality iff \(F\) is a coboundary of the magnetic flow \cite{pollicott94}. Therefore, \(e'(0)=0\) and \(e''(0)\geq 0\) with equality iff there exists \(u\in C^{\infty}(SM)\) such that \((X+\lambda V)u=V(q)\). Now by Theorem~\ref{anosovi2} we know that this occurs iff \(u\) has degree \(1\) in the fibre variable, hence \(V(q)\) is potential. 
\end{proof}

\subsection*{Acknowledgements.}
The author wishes to thank his advisor, Gabriel Paternain, for all his encouragement and support. Many thanks also to Hanming Zhou who alerted the author to an error in an earlier version of this article.

\end{document}